\documentclass[oupthm]{CUP-JNL-BCM}%

\usepackage{amsmath,amssymb,amsfonts}
\usepackage{mathrsfs}
\usepackage{rotating}
\usepackage[numbers]{natbib}

\theoremstyle{oupplain}
\newtheorem*{theorem}{Theorem}

\theoremstyle{oupproof}
\newtheorem{proof}{Proof}

\newcommand{\TT}{{\mathbb T}}
\newcommand{\DD}{{\mathbb D}}

\articletype{RESEARCH ARTICLE}
\jname{Canadian Mathematical Society}
\jyear{2020}

\begin{document}

\begin{Frontmatter}

\title[Decay of singular inner functions]{On the decay of singular inner functions}

\author{Thomas Ransford\thanks{Research supported by grants from NSERC and the Canada research chairs program}}

\address{\orgname{D\'epartement de math\'ematiques et de statistique, Universit\'e Laval}, \orgaddress{\street{Qu\'ebec}, \state{QC}, \postcode{G1V 0A6}}\email{ransford@mat.ulaval.ca}}



\abstract{
It is known that, if $S(z)$ is a non-constant, singular inner function defined on the unit disk, then
$\min_{|z|\le r}|S(z)|\to0$ as $r\to1^-$. We show that the convergence may be arbitrarily slow.
}

\end{Frontmatter}


A \emph{singular inner function} on the unit disk $\DD$ is a function of the form
\begin{equation}\label{E:singinner}
S(z):=\exp\Bigl(-\int_\TT \frac{e^{i\theta}+z}{e^{i\theta}-z}\,d\mu(\theta)\Bigr) \quad(z\in\DD),
\end{equation}
where $\mu$ is a positive finite measure on the unit circle $\TT$ that is singular with respect to linear Lebesgue measure.
Such functions play an important role in the theory of Hardy spaces $H^p$. For example, according to the 
so-called canonical factorization theorem, every function in $H^p (0<p\le\infty)$ that is not identically zero can be factorized in an essentially unique way
as the product of an outer function in $H^p$, a Blaschke product and a singular inner function 
(see e.g.\ \cite[Theorem~1.12]{Ma13}).

As the name suggests, a singular inner function $S$ is inner, i.e.,
it is a bounded holomorphic function whose radial limits satisfy $|S^*(e^{i\theta})|=1$
Lebesgue-a.e. on~$\TT$.
In addition, the fact of being singular implies that
$S$ also satisfies
$S^*(e^{i\theta})=0$ $\mu$-a.e. on~$\TT$
(see e.g.\ \cite[Lemma~1.8]{Ma13}).
In particular, 
if $S$ is non-constant (so $\mu\not\equiv0$), then $S$ has at least one radial limit equal to zero,
and hence
\begin{equation}\label{E:convto0}
\min_{|z|\le r}|S(z)|\to 0 \quad (r\to 1^-).
\end{equation}

Using the formula \eqref{E:singinner}, it is easy to show that
\begin{equation}\label{E:fastconv}
\inf_{0\le r<1}\Bigl(\min_{|z|\le r}|S(z)|\Bigr)^{1-r}>0.
\end{equation}
This imposes a constraint on how rapidly the minimum in \eqref{E:convto0} can converge to zero,
and in fact this constraint is optimal (consider the case when $\mu$ is a point mass).
In this note, we consider how \emph{slowly} the minimum in \eqref{E:convto0} can converge to zero.
The following theorem shows that the convergence may be arbitrarily slow.

\medskip

\begin{theorem}\label{T:mainresult}
Let $\omega:[0,1)\to(0,1)$ be a function such that $\sup_{r\in[0,1)}\omega(r)<1$ and $\lim_{r\to1^-}\omega(r)=0$.
Then there exists a non-constant singular inner function $S$ on $\DD$ such that
\begin{equation}\label{E:estimate}
\min_{|z|\le r} |S(z)|\ge \omega(r) \quad(r\in[0,1)).
\end{equation}
\end{theorem}

\begin{proof}
Replacing $\omega$ by a slightly larger, piecewise-linear function, if necessary,
we may assume that, in addition to its other properties, $\omega$ is a   continuous, decreasing
function on $[0,1)$.

Define two auxiliary functions $g,h:(0,\pi]\to(0,\infty)$ by
\[
g(t):=\frac{1}{3\pi}\log\frac{1}{\omega(1-t/\pi)}
\quad\text{and}\quad
h(t):=\min\Bigl\{\sqrt{t},\inf_{s\in[t,\pi]}sg(s)\Bigr\}.
\]
Then $g$ is a positive, continuous, decreasing function on $(0,\pi]$ with $\lim_{t\to 0^+}g(t)=\infty$,
and $h$ is a positive, continuous, increasing function on $(0,\pi]$ such that $\lim_{t\to0^+}h(t)=0$.
In addition, we clearly have $h(t)\le tg(t)$ for all $t\in(0,\pi]$.
Notice also that, for $t\in(0,1]$,
\[
\frac{h(t)}{t}
=\min\Bigl\{\frac{\sqrt{t}}{t},\inf_{s\in[t,\sqrt{t}]}\frac{sg(s)}{t},\inf_{s\in[\sqrt{t},\pi]}\frac{sg(s)}{t}\Bigr\}
\ge \min\Bigl\{\frac{1}{\sqrt{t}},g(\sqrt{t}),\frac{g(\pi)}{\sqrt{t}}\Bigr\},
\]
and so $\lim_{t\to0^+}h(t)/t=\infty$.

The properties of $h$ (continuous, increasing, $\lim_{t\to0+}h(t)=0$)
make it a so-called measure function,  
which can be used to define an $h$-Hausdorff measure, $\Lambda_h$
(see e.g.\ \cite[Chapter~II]{Ca67}).
We shall consider $\Lambda_h$ on the unit circle $\TT$, 
thought of as a metric space with the arclength distance.

Since $\lim_{t\to 0^+}h(t)/t=\infty$ and $\TT$ has positive linear measure, we have $\Lambda_h(\TT)>0$
(in fact even $\Lambda_h(\TT)=\infty$). By \cite[Chapter~II, Theorem~3]{Ca67}, there exists a closed subset
$E$ of $\TT$ such that $0<\Lambda_h(E)<\infty$. Again using the fact that $\lim_{t\to 0^+}h(t)/t=\infty$,
the inequality $\Lambda_h(E)<\infty$ implies that $E$ is of linear measure zero.

Now we apply a theorem of Frostman, \cite[Chapter~II, Theorem~1]{Ca67}. The fact that $\Lambda_h(E)>0$
implies that there is a finite positive measure $\mu$ on $E$ such that $\mu(E)>0$ and
$\mu([\theta-t,\theta+t])\le h(t)$ for all $\theta\in[-\pi,\pi]$ and $t\in(0,\pi]$.
Since $E$ is of linear measure zero, $\mu$ is singular with respect to Lebesgue measure.
Define $S$ by \eqref{E:singinner}. Then $S$ is a 
non-constant, singular inner function, and we shall show that it satisfies the
estimate \eqref{E:estimate}.

Fix $r\in[0,1)$ and $\theta\in[-\pi,\pi]$. We have
\[
-\log|S(re^{i\theta})|=\int_\TT P_r(t-\theta)\,d\mu(t),
\]
where $P_r(\cdot)$ denotes the Poison kernel, namely
\[
P_r(t-\theta)=\frac{1-r^2}{1+r^2-2r\cos(t-\theta)}.
\]
We estimate the integral by splitting it into two,
according to whether $|t-\theta|$ is smaller than or greater than $\epsilon$. Here $\epsilon\in(0,\pi]$ is
a constant, to be chosen later.
Firstly, we have
\begin{align*}
\int_{|t-\theta|\le\epsilon}P_r(t-\theta)\,d\mu(t)
&\le \int_{|t-\theta|\le\epsilon}\frac{1-r^2}{(1-r)^2}\,d\mu(t)\\
&\le \frac{1-r^2}{(1-r)^2}\mu([\theta-\epsilon,\theta+\epsilon])\\
&\le \frac{1+r}{1-r}h(\epsilon)\le\frac{2}{1-r}\epsilon g(\epsilon).
\end{align*}
Secondly, using integration by parts, we have
\begin{align*}
\int_{|t-\theta|>\epsilon}(P_r(t-\theta)-P_r(\pi))\,d\mu(t)
&=\int_{|t-\theta|>\epsilon}\int_{s=|t-\theta|}^\pi -P_r'(s)\,ds\,d\mu(t)\\
&=\int_{s=\epsilon}^\pi\int_{|t-\theta|\le s} -P_r'(s)\,d\mu(t)\,ds\\
&=\int_{s=\epsilon}^\pi\mu([\theta-s,\theta+s]) (-P_r'(s))\,ds\\
&\le \int_{s=\epsilon}^\pi h(s) (-P_r'(s))\,ds\\
&\le \int_{s=\epsilon}^\pi g(s) (-sP_r'(s))\,ds\\
&\le g(\epsilon)\int_{s=0}^\pi  -sP_r'(s)\,ds\\
&=-\pi P_r(\pi)g(\epsilon)+ g(\epsilon)\int_0^\pi P_r(s)\,ds \\
&= -\pi P_r(\pi)g(\epsilon)+\pi g(\epsilon),
\end{align*}
whence
\begin{align*}
\int_{|t-\theta|>\epsilon}P_r(t-\theta)\,d\mu(t)
&\le -\pi P_r(\pi)g(\epsilon)+\pi g(\epsilon)+P_r(\pi)\mu(\TT)\\
&\le -\pi P_r(\pi)g(\epsilon)+\pi g(\epsilon)+P_r(\pi)\pi g(\pi)\le \pi g(\epsilon).
\end{align*}
Putting these estimates together, we obtain
\[
-\log|S(re^{i\theta})|\le \Bigl(\frac{2\epsilon}{1-r}+\pi\Bigr)g(\epsilon).
\]
In particular, taking $\epsilon=\pi(1-r)$, we get
\[
-\log|S(re^{i\theta})|\le 3\pi g(\pi(1-r))=\log\frac{1}{\omega(r)}.
\]
Rearranging this equality and minimizing over $\theta$ yields
\[
\min_{|z|= r}|S(z)|\ge \omega(r),
\]
which is equivalent to \eqref{E:estimate}, by the minimum principle.
\end{proof}

\textbf{Acknowledgement.}
The theorem above was proved in order to solve a problem posed to me by Joe Cima.
I am grateful to Joe for drawing my attention to this question.

\begin{Backmatter}

\printaddress

\end{Backmatter}

\end{document}